\documentclass[11pt]{amsart}

\usepackage{amsmath, amsthm, amssymb}
 
\usepackage{palatino}         
\usepackage{comment}

\usepackage[abs]{overpic}		  

\usepackage{xcolor}
\definecolor{indigo}{rgb}{0.29, 0.0, 0.51}  
\usepackage[colorlinks, urlcolor=indigo, linkcolor=indigo, citecolor=indigo]{hyperref}

\theoremstyle{plain}
\newtheorem{theorem}{Theorem}
\newtheorem{corollary}[theorem]{Corollary}
\newtheorem{proposition}[theorem]{Proposition}
\newtheorem{lemma}[theorem]{Lemma}

\theoremstyle{definition}

\theoremstyle{remark}
\newtheorem{remark}[theorem]{Remark}

\numberwithin{theorem}{section}

\newcommand{\dfn}[1]{{\em #1}}        
\newcommand{\R}{\mathbb{R}}           
\newcommand{\Z}{\mathbb{Z}}           
\DeclareMathOperator{\Ima}{Im}        

\makeatletter
\newcommand*\bigcdot{\mathpalette\bigcdot@{0.6}}
\newcommand*\bigcdot@[2]{\mathbin{\vcenter{\hbox{\scalebox{#2}{$\m@th#1\bullet$}}}}}
\makeatother

\begin{document}

\title[Lagrangian Cobordisms]{On Lagrangian Cobordisms And the Chekanov-Eliashberg DGA}

\author{Sierra Knavel}

\author{Thomas Rodewald}

\address{School of Mathematics \\ Georgia Institute of Technology \\ Atlanta, GA}
\email{sknavel3@gatech.edu} 

\address{School of Mathematics \\ Georgia Institute of Technology \\  Atlanta, GA}
\email{tomrodewald@gatech.edu}


\begin{abstract}
In this paper, we consider exact Lagrangian cobordisms and the map they induce on the Chekanov-Eliashberg DGAs of their Legendrian ends as defined by Ekholm, Honda, and Kalman. Specifically, we show how to adapt this map to linearizations of the DGA using augmentations. We then show its induced map on linearized Legendrian contact homology is invariant under Lagrangian isotopy under mild hypotheses, as well as its induced map on higher order product structures. 
\end{abstract}

\maketitle

\section{Introduction}
Lagrangian cobordisms between Legendrian knots are a central area of study in contact and symplectic geometry and have been intensely studied in the past couple of decades. There are many examples of constructions of such cobordisms \cite{BourgeoisSabloffTraynor, CornwellNgSivek, EckholmHondaKalman, GuadagniSabloffYacavone, Lin16}, as well as obstructions to their existence \cite{Chantraine_10, CornwellNgSivek, EkholmEtnyreSullivan, Gromov}. In this paper, we study Lagrangian cobordisms and, in particular, the Differential Graded Algebra (DGA) map they induce on linearized Legendrian contact homology. The existence of such a map was shown in \cite{Pan}, and a similar version in the language of generating family homology was also shown in \cite{SabloffTraynor}. Naturally one might ask how isotoping an exact Lagrangian cobordism through Lagrangians effects this map. After explicitly defining the DGA map on the chain level, we show under mild hypotheses that its induced map on the homology is invariant under Lagrangian isotopy.

To establish this invariance, consider the Chekanov-Eliashberg DGA. Chekanov \cite{Chekanov_DGA_02} and Eliashberg \cite{Eliashberg1998} independently define a DGA $(\mathcal{A}_\Lambda, \partial)$ associated to a Legendrian $\Lambda$ in $(\R, \xi_{std})$ to distuinguish Legendrian knots which cannot be distinguished by their classical invariants $tb$ and $r$ \cite{Chekanov_DGA_02}. Since this result, the DGA has been an instrumental tool in the study of Legendrian knots, see \cite{EpsteinFuchsMeyer, EtnyreNgSabloff, NgRutherford, NgCharAlgebra, Sabloff}. Because the DGA, and by extension its homology, is a non-commutative ring, its computation is difficult. However, in some cases one can distinguish pairs of Legendrian knots by linearizing the DGA using augmentations. 

Chekanov \cite{Chekanov_DGA_02} defines an augmentation as a chain map 
$$\epsilon: (\mathcal{A}_\Lambda, \partial) \rightarrow (\Z/2\Z,0)$$ 
in order to modify the differential of the DGA to a new differential called $\partial_{\epsilon}$. One can then `linearize' $\partial_{\epsilon}$ to obtain a differential $\partial_{\epsilon}^\ell$ on a finite-dimensional vector space $F_\Lambda$. The homology of the chain complex yielded by this differential is the \dfn{linearized contact homology} of $\Lambda$. For a given Legendrian $\Lambda$, if an augmentation exists, then the linearized contact homology is easier to compute and the set of all linearized contact homologies is invariant up to Legendrian isotopy. The base ring $\Z/2\Z$ has since been generalized to a unital ring \cite{EtnyreNgSabloff}; in this paper we use $\Z$. 

Let $L$ be an exact Lagrangian cobordism $L$ from a Legendrian knot $\Lambda_1$ to a Legendrian $\Lambda_2$. Eckholm, Honda, and Kalman \cite{EckholmHondaKalman} define the following chain map between their corresponding DGAs.
$$\Phi_L: (\mathcal{A}_{\Lambda_2}, \partial_{\Lambda_2}) \to (\mathcal{A}_{\Lambda_1}, \partial_{\Lambda_1})$$ 
Once again, the non-commutativity of this map makes computations difficult, and therefore we aim to linearize $\Phi_L$, show that it induces a homomorphism on the linearized contact homology, and apply it to examples through direct computation. That is, given an augmentation $\epsilon_1$ for $\Lambda_1$, we consider the augmentation $\epsilon_1 \circ \Phi_L$ of $\Lambda_2$ and then define a map between the complexes $F_{\Lambda_i}$ (see Section~\ref{Sec:cobordism_map}),
$$(\Phi_L^{\epsilon_1})^\ell: F_{\Lambda_2} \rightarrow F_{\Lambda_1}.$$

We show that this gives a chain map, thus inducing a homomorphism on the linearized contact homology.  

\begin{proposition}\label{DGA} Let $L$ be an exact Lagrangian cobordism from $\Lambda_1$ to $\Lambda_2$, and $\epsilon_1$ an augmentation of $(\mathcal{A}_{\Lambda_1}, \partial_{\Lambda_1})$ and $\epsilon_2 = \epsilon_1 \circ \Phi_L$. Then
$$(\Phi_L^{\epsilon_1})^\ell \circ \partial_{\epsilon_2}^\ell = \partial_{\epsilon_1}^\ell \circ (\Phi_L^{\epsilon_1})^\ell.$$
\end{proposition}

We remind the reader the existence of this map was shown in \cite{Pan}, and in different language in \cite{SabloffTraynor}. In order to use $(\Phi_L^{\epsilon_1})^\ell$ to study Lagrangian cobordisms, we need to prove some basic facts about the map and its behavior under geometric phenomena, such as Lagrangian isotopy. Essentially, we have a linearized version of \cite[Theorems~1.3 part 1) and 1.4]{EckholmHondaKalman}.

\begin{theorem}\label{LinearizedEHK}
   The DGA map induced by the exact Lagrangian cobordism $L$ from $\Lambda_1$ to $\Lambda_2$, $$(\Phi_L^{\epsilon_1})^\ell: (F_2, \partial_{\epsilon_2}^\ell) \rightarrow (F_1, \partial_{\epsilon_1}^\ell),$$ has the following properties: \\
    \indent 1) If $L = \Lambda \times \R$, then $(\Phi_L^{\epsilon_1})^\ell = Id_{\mathcal{A}_\Lambda}$. \\
    \indent 2) Let $L_1$ and $ L_2$ be exact Lagrangian cobordisms from $\Lambda_1$ and $\Lambda_2$ that are isotopic through exact Lagrangian cobordisms. If $\epsilon_1 \circ \Phi_{L_1} = \epsilon_1 \circ \Phi_{L_2}$ and the image of $\partial_{\epsilon_2}$ has no quadratic terms with Reeb chords of grading $-1$, then $(\Phi_{L_1}^{\epsilon_1})^\ell$ and $(\Phi_{L_2}^{\epsilon_1})^\ell$ are chain homotopic.\\
\end{theorem}

\begin{remark}
 It is natural to think that given that the map $\Phi_L$ changes by a chain homotopy equivalence when $L$ undergoes an exact Lagrangian isotopy, then the same would be true for the linearized maps, but the above theorem shows that extra information is needed to do this. In particular, we need extra hypotheses on the DGAs involved to guarantee that we can define maps on their linearizations that are invariant under exact Lagrangian isotopy (see Remark~\ref{LinearizedThmRemark}). We do not know if these hypotheses are necessary, but it is not clear how one can establish the desired properties of $(\Phi_L^{\epsilon_1})^\ell$ without them.
\end{remark}

    Given an augmentation $\epsilon$ of the DGA for a Legendrian knot $\Lambda$, one may define a product structure, in fact an $A_\infty$ structure, on the dualized linear contact homology \cite{CivanEtnyreKoprowskiSabloffWalker}. It was shown in \cite{CivanEtnyreKoprowskiSabloffWalker} that the product structure and the entire $A_\infty$ structure is a stronger invariant of a Legendrian knot than just the linearization alone. Pan showed the existence of an $A_\infty$ morphism between the $A_\infty$ structures on their Legendrian ends in \cite{Pan}, and we replicate this result using our argument on the chain level.

\begin{theorem}
\label{AinftyTHM}
Let $(F_1^*, \{m_k^{\epsilon_1}\}_{k \geq 1})$ and $(F_2^*, \{m_k^{\epsilon_2}\}_{k \geq 1})$ be the $A_\infty$ structures on the linearized Legendrian cochain groups. Then the collection of maps $\{((\Phi_L^{\epsilon_1})^k)^*: (F_1^*)^{\otimes k} \rightarrow F_2^*\}_{k \geq 1}$ is an $A_\infty$ morphism.
\end{theorem}

An implication of Theorem~\ref{AinftyTHM} is that $((\Phi_L^{\epsilon_1})^\ell)^*$ (which is $((\Phi_L^{\epsilon_1})^1)^*$ in the notation in Theorem~\ref{AinftyTHM} above) preserves the higher order product structures described in Section~\ref{Ainftybackground}. As a corollary to Theorem~\ref{LinearizedEHK} and Theorem~\ref{AinftyTHM}, under certain conditions, we have invariance of our linearized augmented DGA maps respecting these structures under Lagrangian cobordism. 

\begin{corollary}
    Let $L_1$ be Lagrangian isotopic to $L_2$, $\epsilon_1$ an augmentation for $\Lambda_1$, $\epsilon_1 \circ \Phi_{L_1} = \epsilon_1 \circ \Phi_{L_2}$, and the quadratics in the image of $\partial_{\epsilon_2}$ have no degree $-1$ Reeb chords. Then the induced maps of $((\Phi_{L_1}^{\epsilon_1})^\ell)^*$ and  $((\Phi_{L_2}^{\epsilon_1})^\ell)^*$ are equal on the higher order product structures. 
\end{corollary}

Since Pan proved exact Lagrangian cobordism induces an $A_\infty$ morphism, Sabloff and Ma have shown the $A_\infty$ structures of its Legendrian ends can be equipped with a \dfn{weak relative Calabi-Yau} structure \cite{ma2025weakrelativecalabiyaustructures}, and Ng has shown the existence of an $L_\infty$ structure on Legendrian contact homology \cite{ng2025linfinitystructurelegendriancontact}. The authors hope to show in future work that $\Phi_L$ can be modified to preserve these structures as well.

Another useful algebraic object extracted from the DGA of a Legendrian knot is the \dfn{characteristic algebra} defined by Ng in \cite{NgCharAlgebra}. There it is shown that the characteristic algebra is a Legendrian isotopy class invariant. This invariant is used to distinguish Legendrian representatives of the same smooth knot type. It is easy to show that $\Phi_L$ also induces a well-defined map on the characteristic algebra. In future work, the authors hope to show that this induced map is also invariant under Lagrangian isotopy. 
\\
\\
\noindent {\bf Acknowledgements:} The authors would like to thank John Etnyre for his helpful conversations and advisement. We also thank Josh Sabloff and Lenny Ng for helpful discussions and pointing out references \cite{Pan, SabloffTraynor}. This work is partially supported by National Science Foundation grant DMS-2203312.
\section{Background}

We assume a reader is familiar with fundamental notions of contact and symplectic geometry, such as in \cite{Etnyre_05survey}, and the basics of Legendrian contact homology, such as in \cite{Etnyre_Ng_22survey}. We recall some of this material here for convenience and to establish notation. In Section~\ref{Sec:prelims}, we recall basic notions about Legendrian knots and the definition of an exact Lagrangian cobordism. In Section~\ref{Sec:CE_DGA}, we discuss the Chekanov-Eliashberg DGA. And finally, in Section~\ref{Sec:EHK}, we recall the map of Eckholm, Honda, and Kalman on the DGA induced from an exact Lagrangian cobordism.

\subsection{Preliminaries}
\label{Sec:prelims}
 
\begin{figure}[h]
    \centering
    \def\svgwidth{\linewidth} 
    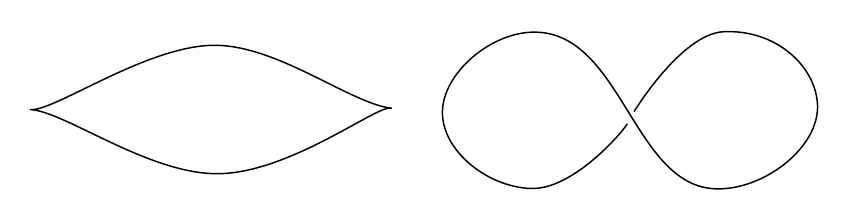
    \caption{The front projection of the max-tb unknot, left, and its Lagrangian projection, right.}
    \label{UnknotProjections}
\end{figure}

In $\R^3$, the standard contact structure $\xi_{std}$ is a 2-dimensional plane field given by the kernel of the 1-form $\alpha = dz - ydx$. This will be the only contact structure used in this paper. Legendrian knots in $(\R^3, \xi_{std})$ can be seen in two ways: the \dfn{front projection}, which projects the knot onto the $xz$-plane, that is $\Pi: \R^3 \rightarrow \R^2: (x,y,z)\mapsto (x,z)$, and the \dfn{Lagrangian projection}, which projects to the $xy$-plane, that is $\pi: \R^3 \rightarrow \R^2: (x,y,z)\mapsto (x,y)$. Legendrian knots in $(\R^3, \xi_{std})$ can be distinguished by two classical invariants, the \dfn{Thurston-Bennequin} ($tb$) and the \dfn{rotation number} ($r$). 
For more information on Legendrian knots and the computations of these classical invariants, we refer the reader to the following survey article \cite{Etnyre_05survey}. The Legendrian knots (and links) in this paper are closed and oriented unless otherwise stated.

An \dfn{exact symplectic manifold} is a symplectic manifold $(M^{2n}, \omega)$ equipped with  a 1-form $\lambda$ such that $\lambda = d\omega$.
An $n$-dimensional submanifold $L$ in $M$ is \dfn{Lagrangian} if $\omega|_L=0$. The Lagrangian is \dfn{exact} if it also holds that $\lambda|_L$ is an exact 1-form. One can consider \dfn{exact Lagrangian cobordisms}, a particularly interesting class of non-closed, exact Lagrangian submanifolds of $\left(\R \times \R^3, d(e^t\alpha)\right)$, where $t$ is the coordinate the first $\R$-factor. We call this pair the \dfn{symplectization} of $(\R^3, \xi_{std})$ and it is an exact symplectic manifold. The symplectization has cylindrical ends over Legendrians. Lagrangian cobordisms are assumed to be cobordism from Legendrian $\Lambda_1$ to Legendrian $\Lambda_2$, see Figure~\ref{sketch_cobordism}. Lagrangians are not as abundant as Legendrian knots, and the subset of exact Langrangian cobordisms are even less so. 

\begin{figure}[h!]
  \centering
   \begin{overpic}[width=3.5in]{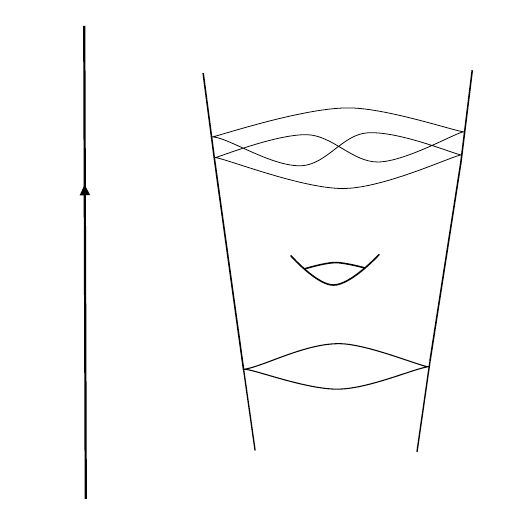}
   \put(50, 220){$t$}
    \put(250, 75){$\Lambda_1$}
    \put(250, 180){$\Lambda_2$}
   \end{overpic}
       \caption{A schematic of a Lagrangian cobordism from a Legendrian unknot to a Legendrian trefoil in the symplectization of $(\R^3, \xi_{std})$.}
       \label{sketch_cobordism}
\end{figure}

Formally, an exact Lagrangian cobordism $L$ from $\Lambda_1$ to $\Lambda_2$ is an embedded, orientable Lagrangian surface in $\left( \R \times \R^3, d \left( e^t \alpha \right) \right)$ such that for some $T \in \R^+$, the following hold:
\begin{itemize}
    \item $L \cap \left( [-T,T] \times \R^3\right)$ is compact
    \item $L \cap \left( \left(T, \infty \right) \times \R^3\right) = (T, \infty) \times \Lambda_2$
    \item $L \cap \left( \left(-\infty, -T \right) \times \R^3\right) = (-\infty, -T) \times \Lambda_1$, and 
    \item there exists a function $f: L \rightarrow \R$ which is constant on each end of $L$ whenever $df = \lambda|_L$.
\end{itemize}

Given a Lagrangian cobordism $L$ between two Legendrian knots $\Lambda_1$ and $\Lambda_2$, Chantraine \cite{Chantraine_10} shows that the difference in their Thurston-Bennequin numbers relates to the Euler characteristic of the cobordism and their rotations numbers are the same. That is, $$tb(\Lambda_2) - tb(\Lambda_1) = - \chi(L),$$ and $$r(\Lambda_1) = r(\Lambda_2).$$ In the case of a concordance, a genus-zero cobordism, the Euler characteristic $\chi(L) = 0$ implies that $tb(\Lambda_1)=tb(\Lambda_2)$. A reader who wishes to learn more about Lagrangian cobordisms is directed to \cite{Blackwell_legout_etal_21survey}.
%
\subsection{Chekanov-Eliashberg DGA}
\label{Sec:CE_DGA}
The classical invariants $tb$ and $r$ associated to Legendrian isotopy classes of Legendrian knots are used to distinguish some but not all of the knots. However, some non-classical invariants can distinguish them, including \dfn{Legendrian contact homology} (LCH), the homology of the Chekanov-Eliashberg differential graded algebra (DGA), and the \dfn{linearized contact homology}. The Chekanov-Eliashberg DGA was independently developed by Eliashberg \cite{Eliashberg1998} and Chekanov \cite{Chekanov_DGA_02}, who showed that the DGAs associated to $5_2$ and its mirror are not stable tame isomorphic and hence not Legendrian isotopic. Here, a stable tame isomorphism is a notion of equivalence between DGAs \cite{Chekanov_DGA_02}. We obscure the definition of stable tame isomorphism and instead highlight the fact that two Legendrian isotopic knots will have equivalent DGAs. The Chekanov-Eliashberg DGA of a Legendrian knot associates to every Legendrian knot $\Lambda \in (\R^3, \xi_{std})$ a differential graded algebra with generators the \dfn{Reeb chords}, or the trajectories of the Reeb vector field whose start and end points are on $\Lambda$ \cite{Chekanov_DGA_02}. 

Let $\Lambda$ be a Legendrian knot. The Chekanov-Eliashberg DGA, which we denote $\mathcal{A}_\Lambda$, was originally an algebra over $\Z/2\Z$ but now refers to its lifting to an algebra over the ring $\Z[t,t^{-1}]$ with grading over $\Z$ if $r(\Lambda) = 0$. The algebra is a free associated unital algebra generated by the finite set of Reeb chords $\mathcal{C}_\Lambda = \{a_1, a_2, \ldots, a_n \}$. That is, $\mathcal{A}_\Lambda = \Z[t,t^{-1}]\langle a_1, a_2, \ldots, a_n \rangle$. Up to $2r\left(\Lambda\right)$, the grading on $\mathcal{A}_\Lambda$ is defined for the generators. We describe it in the Lagrangian projection $\Pi(\Lambda)$, where the Reeb chords of $\Lambda$ are the double points in the projection. Let $a_i \in \mathcal{C}_\Lambda$ and $*$ a basepoint on $\Lambda$. Let $\lambda_i$ be a path from the overcrossing of $a_i$ to its undercrossing and missing $*$. $\Pi(\Lambda)$ can be perturbed such that the fractional number of counterclockwise rotations of the tangent vector to $\lambda_i$, denoted $r(\lambda_i)$, is a multiple of $\frac{1}{4}$. Then the grading on $a_i$ is $$|a_i| = 2r(\lambda_i) - \frac{1}{2}.$$ The gradings of $t$ and $t^{-1}$ are $-2r(\Lambda)$ and $2r(\lambda)$, respectively. The grading of a word in $\mathcal{A}_\Lambda$ is then the sum of the gradings of the generators in the word. 

For a complete combinatorial description of the differential $\partial$ using the Lagrangian projection, we refer the reader to \cite{Etnyre_Ng_22survey}. In this paper, we will use that it satisfies the Leibniz rule, i.e. for any two words $w$ and $w'$ in $\mathcal{A}_\Lambda$, $$\partial(ww') = \partial(w)w' + (-1)^{|w|}w \partial(w').$$

Due to the infinite dimensionality of the Chekanov-Eliashberg DGA, its homology may also be infinite dimensional. This hurdle, in addition to the noncommutativity of the algebra, makes computations difficult. To overcome this, Chekanov \cite{Chekanov_DGA_02} considers DGAs which have \dfn{augmentations} in order to work with a finite-dimensional linear complex whose homology, called linearized contact homology, is a finite-dimensional vector space.

The chain groups are free modules over the ring $\Z[t,t^{-1}]$ generated by the Reeb chords with the same grading used in the Chekanov-Eliashberg DGA. Let $\epsilon: (\mathcal{A}_\Lambda, \partial) \to (\Z, 0)$ be an augmentation. The map $\phi^\epsilon: \mathcal{A}_\Lambda \to \mathcal{A}_\Lambda$, called a \dfn{tame isomorphism}, is defined by its action on Reeb chords. Let $c \in \mathcal{C}_\Lambda$. Then  $$\phi^{\epsilon}(c) = c + \epsilon(c),$$ while fixing constants. Using tame isomorphisms we define the boundary of the complex, denoted $\partial_\epsilon^\ell$, to be the linear terms of $\partial_\epsilon = \phi^\epsilon \circ \partial \circ \phi^{-\epsilon}$. Given a Legendrian $\Lambda$ and an augmentation $\epsilon$, we denote the complex by $(F_\Lambda, \partial_\epsilon^\ell)$. The linearized contact homology, which we denote $LCH_*^\epsilon(\Lambda)$, is not itself an invariant. However, Chekanov proved the following theorem.
\begin{theorem}
    \label{ChekanovThm}
    (Chekanov \cite{Chekanov_DGA_02})
    The collection $\{LCH_*^\epsilon(\Lambda): \epsilon \text{ is an augmentation } (\mathcal{A}_\Lambda, \partial) \to (\Z, 0) \}$ is an invariant of $\Lambda$ up to Legendrian isotopy.
\end{theorem}
For more details of the construction of $LCH^\epsilon_*(\Lambda)$ see \cite{Etnyre_Ng_22survey}.

\subsection{Chekanov-Eliashberg DGA cobordism map} 
\label{Sec:EHK}
Given an exact Lagrangian cobordism $L$ between two Legendrians, Ekholm, Honda, and Kalman \cite{EckholmHondaKalman} introduced a DGA map on their Chekanov-Eliashberg DGAs. They further established the following result. 
\begin{theorem}
\label{EHK}
    (Ekholm-Honda-Kalman \cite[Theorems~1.3 and 1.4]{EckholmHondaKalman}) An exact Lagrangian cobordism $L$ from $\Lambda_1$ to $\Lambda_2$ induces a DGA map $\Phi_L: (\mathcal{A}_2, \partial_2) \rightarrow (\mathcal{A}_1, \partial_1)$ with the following properties: \\
    \indent 1) If $L = \Lambda \times \R$ is a trivial Lagrangian cylinder, then $\Phi_L = Id_{\mathcal{A}_\Lambda}$. \\
    \indent 2) If $L_1, L_2$ have the same ends $\Lambda_1$ and $\Lambda_2$ and are isotopic through exact Lagrangian cobordisms, then $\Phi_{L_1}$ and $\Phi_{L_2}$ are chain homotopic.\\
    \indent 3) If $L_1$ and $L_2$ are exact Lagrangian cobordisms from $\Lambda_0$ to $\Lambda_1$ and $\Lambda_1$ to $\Lambda_2$ respectively, then $\Phi_L$ and $\Phi_{L_1} \circ \Phi_{L_2}$ are chain homotopic.
\end{theorem}

\subsection{The $A_\infty$ structure}
\label{Ainftybackground}
Let $V$ be a graded vector space over $\Z/2\Z$. An \dfn{$A_\infty$ algebra} on $V$ is an infinite sequence of degree 1 graded maps $\{m_k: V^{\otimes k} \rightarrow V\}_{1 \leq k}$ that satisfy:

$$\sum_{i+j+k=l} m_{i+1+k} \circ (1^{\otimes i} \otimes m_j \otimes 1^{\otimes k}) = 0$$
for all $1 \leq l \leq n$. The case $l=1$ yields the equation $m_1 \circ m_1 = 0$, so $m_1$ is a codifferential on $V$. The case $l=2$ gives the equation
$$m_1 \circ m_2 (a,b) = m_2( m_1(a),b) + m_2(a, m_1(b)).$$
So $m_2$ descends to a well-defined product on $H^*(V)$. A full $A_\infty$ structure cannot be defined on $H^*(V)$ for the maps $m_k$, $k \geq 3$, by simply letting $m_k$ descend to the cohomology. Instead the higher order $(k \geq 3)$ products are inductively defined on quotients of the cohomology. See  \cite{CivanEtnyreKoprowskiSabloffWalker} for details.
It was shown in \cite{CivanEtnyreKoprowskiSabloffWalker} that there is an $A_\infty$ algebra on $(F_\Lambda^*,(\partial_\epsilon^\ell)^*)$.

\begin{theorem}
(Civan-Etnyre-Koprowski-Sabloff-Walker  \cite{CivanEtnyreKoprowskiSabloffWalker})
    For each augmentation $\epsilon$, the Legendrian contact homology DGA $(\mathcal{A}_\Lambda, \partial_\Lambda)$ induces an $A_\infty$ structure on the linearized cochain complex $(F_\Lambda^*, (\partial_\epsilon^\ell)^*)$.
\end{theorem}

Let $(V, \{m_k: V^{\otimes k} \rightarrow V\}_{1 \leq k})$ and $(W, \{n_k: W^{\otimes k} \rightarrow W\}_{1 \leq k})$ be $A_\infty$ algebras. An \dfn{$A_\infty$ morphism} is an infinite collection of degree 0 maps $\{\phi_n: V^{\otimes n} \rightarrow W\}_{n \geq 1}$ that satisfy

$$\sum_{i+j+k=n} \phi_{i+1+k} \circ (1^{\otimes i} \otimes m_j \otimes 1^{\otimes k}) = \sum_{r=1}^n \sum_{i_1+\ldots + i_r} n_r \circ (\phi_{i_1} \otimes \cdots \otimes \phi_{i_r})$$
for all $n \geq 1$. A consequence of  $\{\phi_n: V^{\otimes n} \rightarrow W\}_{n \geq 1}$ being an $A_\infty$ morphism is that $\phi_1$ preserves the product structures, i.e. $$\phi_1 \circ m_k([c_{i_1}], \ldots, [c_{i_k}]) = n_k(\phi_1[c_{i_1}], \ldots, \phi_1[c_{i_k}])$$ on the level of cohomology.


\section{Linearizing the cobordism map}\label{Sec:cobordism_map}
In this section, we show how to obtain a map on the linearized contact homology of Legendrian knots from an exact Lagrangian cobordism. We then explore the properties of this map when it is a trivial cylinder or altered through exact Lagrangian cobordisms by an isotopy. 

Let $L$ be an exact Lagrangian cobordism from $\Lambda_1$ to $\Lambda_2$, $\Phi_L$ its induced DGA map. Let  $\epsilon_1: (\mathcal{A}_{\Lambda_1}, \partial_1) \rightarrow (\Z, 0)$ be an augmentation and $\epsilon_2 = \epsilon_1 \circ \Phi_L$. We define the \dfn{augmented DGA map}, $\Phi_L^{\epsilon_1}: (\mathcal{A}_{\Lambda_2}, \partial_{\epsilon_2}) \rightarrow (\mathcal{A}_{\Lambda_1}, \partial_{\epsilon_1})$, to be $\Phi_L^{\epsilon_1} = \phi^{\epsilon_1}  \circ \Phi_L \circ \phi^{-\epsilon_2}$. Recall $\phi^{\epsilon_i}: \mathcal{A}_{\Lambda_i} \rightarrow \mathcal{A}_{\Lambda_i}$ was defined in Section~\ref{Sec:CE_DGA}. We first show that $\Phi_L^{\epsilon_1}$ is a chain map from $(\mathcal{A}_{\Lambda_2}, \partial_{\epsilon_2})$ to $(\mathcal{A}_{\Lambda_1}, \partial_{\epsilon_1})$, where $\partial_{\epsilon_i} = \phi^{\epsilon_i} \circ \partial \circ \phi^{-\epsilon_i}, i=1,2$, which we refer to as the \dfn{augmented differentials}.
 
\begin{lemma}
\label{AugCommute}
Let $L$ be an exact Lagrangian cobordism from $\Lambda_1$ to $\Lambda_2$. Then
$\Phi_L^{\epsilon_1} \circ \partial_{\epsilon_2} = \partial_{\epsilon_1} \circ \Phi_L^{\epsilon_1}$.
\end{lemma}
\begin{proof}
A computation gives
    \begin{equation*}
        \begin{split}
            \Phi_L^{\epsilon_1} \circ \partial_{\epsilon_2} &= \phi^{\epsilon_1}  \circ \Phi_L \circ \phi^{-\epsilon_2} \circ \phi^{\epsilon_2} \circ \partial_2 \circ \phi^{-\epsilon_2}  \\
            &= \phi^{\epsilon_1} \circ \Phi_L \circ \partial_2 \circ \phi^{-\epsilon_2} \\
            &= \phi^{\epsilon_1} \circ \partial_1 \circ \Phi_L \circ \phi^{-\epsilon_2} \\
            &= \phi^{\epsilon_1} \circ \partial_1 \circ \phi^{-\epsilon_1} \circ \phi^{\epsilon_1} \circ \Phi_L \circ \phi^{-\epsilon_2} \\
            &= \partial_{\epsilon_1} \circ \Phi_L^{\epsilon_1},
        \end{split}
    \end{equation*}
    where the third equality follows from Theorem$~\ref{EHK}$
    Thus we see that $\Phi_L^{\epsilon_1}$ commutes with the augmented differentials.
\end{proof}
In order to 'linearize' the equation in Lemma~\ref{AugCommute}, we will need the following two lemmas. The first shows that $\Phi_L^{\epsilon_1}$ has no constant terms.
\begin{lemma}
\label{NoConstants}
Let $c$ be a Reeb chord. Then $\Phi_L^{\epsilon_1}(c)$ has no constant terms.
\end{lemma}
\begin{proof} Let $\mathcal{C}_{\Lambda_1} = \{c_1, \ldots,c_m\}$ be the Reeb chords of $\Lambda_1$, $c$ a Reeb chord of $\Lambda_2$. We start by showing the constant terms of 
$\phi^{\epsilon_1} \circ \Phi_L(c)$ are $\epsilon_1 \circ \Phi_L(c)$. There are constants $a_0$, $a_{i_1, \ldots, i_j} \in \Z$ for $j \geq 1$ and $i_k \in \{1, \ldots, n\}$, with only finitely many being nonzero such that $$\Phi_L(c) = a_0 + \sum_{j=1}^\infty \left(\sum_{i_1, \ldots, i_j} a_{i_1, \ldots, i_j} c_{i_1} \cdots c_{i_j} \right).$$ Then $$\phi^{\epsilon_1} \circ \Phi_L(c) = a_0 + \sum_{j=1}^\infty \left(\sum_{i_1, \ldots, i_j} a_{i_1, \ldots, i_j} (c_{i_1}+\epsilon(c_{i_1})) \cdots (c_{i_j}+ \epsilon_1(c_{i_j})) \right).$$ The constant terms of this sum are $$\phi^{\epsilon_1} \circ \Phi_L(c) = a_0 + \sum_{j=1}^\infty \left(\sum_{i_1, \ldots, i_j} a_{i_1, \ldots, i_j}\epsilon(c_{i_1}) \cdots \epsilon_1(c_{i_j}) \right) = \epsilon_1 \circ \Phi_L(c).$$ \\
\indent Now we show $\Phi_L^{\epsilon_1}(c)$ has no constant terms. We have \begin{equation*}
\begin{split} 
\Phi_L^{\epsilon_1}(c) &= \phi^{\epsilon_1} \circ \Phi_L \circ \phi^{-\epsilon_2}(c) \\ &=  \phi^{\epsilon_1} \circ \Phi_L(c - \epsilon_2(c)) \\ &= \phi^{\epsilon_1} (\Phi_L(c) - \Phi_L(\epsilon_2(c))) \\ &= \phi^{\epsilon_1}( \Phi_L(c) - (\epsilon_2(c))) \\ &= \phi^{\epsilon_1} (\Phi_L(c) - (\epsilon_1 \circ \Phi_L(c))) \\ &= \phi^{\epsilon_1} \circ \Phi_L(c) - \epsilon_1 \circ \Phi_L(c).
\end{split}
\end{equation*}
As shown in the previous paragraph, the constant terms of $\phi^{\epsilon_1} \circ \Phi_L(c)$ are exactly $\epsilon_1 \circ \Phi_L(c).$

\end{proof}

It is well-known that the augmented differentials satisfy the Leibniz rule. We provide a proof of this here as we will need it below, and it does not appear in the literature.

\begin{lemma}
\label{Leibniz}
    The augmented boundary map $\partial_{\epsilon}$ satisfies the Leibniz rule, that is $$\partial_\epsilon(c_1 \cdots c_m) = \partial_\epsilon(c_1)c_2 \cdots c_m + (-1)^{|c_1|}c_1\partial_\epsilon(c_2)c_3 \cdots c_m + \ldots + (-1)^{|c_1 \cdots c_{m-1}|}c_1 \cdots c_{m-1}\partial_\epsilon (c_m).$$
\end{lemma}
\begin{proof}
We first show $$\partial(\phi^{-\epsilon}(c_1 \cdots c_m)) = \\ \sum_{i=1}^m (-1)^{|c_1 \cdots c_{i-1}|}\phi^{-\epsilon}(c_1 \cdots c_{i-1})\partial(\phi^{-\epsilon}(c_i))\phi^{-\epsilon}(c_{i+1} \cdots c_m)$$ 
by induction. Assume true for $m-1$. Then for a word of length $m$ we have 
\begin{equation*}
    \begin{split}
\partial(\phi^{-\epsilon}(c_1 \cdots c_m)) &= \partial((c_1 - \epsilon_1(c_1))\phi^{-\epsilon}(c_2 \cdots c_m))
\\ &= \partial(c_1\phi^{-\epsilon}(c_2 \cdots c_m)) - \epsilon(c_1) \partial(\phi^{-\epsilon}(c_2 \cdots c_m)).  
\end{split}
\end{equation*}
Let $\sum_{i=1}^n w_i$ be the sum of words of length at most $m-1$ after expanding $\phi^{-\epsilon}(c_2 \cdots c_m)$. Then 
\begin{equation*}
    \begin{split}
\partial(c_1\phi^{-\epsilon}(c_2 \cdots c_m)) &= \partial(c_1\sum_{i=1}^n w_i) \\ &= \sum_{i=1}^n \partial (c_1 w_i) \\ &= \sum_{i=1}^n \partial (c_1)w_i + (-1)^{|c_1|}c_1 \partial(w_i) \\ &= \partial(c_1)\sum_{i=1}^n w_i + (-1)^{|c_1|}c_1\partial(\sum_{i=1}^n w_i).
\end{split}
\end{equation*}
Thus, 
\begin{equation*}
    \begin{split}
& \partial(c_1\phi^{-\epsilon}(c_2 \cdots c_m)) - \epsilon(c_1) \partial(\phi^{-\epsilon}(c_2 \cdots c_m)) \\ &= \partial(c_1)\phi^{-\epsilon}(c_2 \cdots c_m) + (-1)^{|c_1|}c_1\partial(\phi^{-\epsilon}(c_2 \cdots c_m))
 - \epsilon(c_1) \partial(\phi^{-\epsilon}(c_2 \cdots c_m)) \\ &= \partial(\phi^{-\epsilon}(c_1))\phi^{-\epsilon}(c_2 \cdots c_m) + (-1)^{|c_1|}\phi^{-\epsilon}(c_1)\partial(\phi^{-\epsilon}(c_2 \cdots c_m)),
\end{split}
\end{equation*}
where $(-1)^{|c_1|}c_1-\epsilon(c_1) = (-1)^{|c_1|}\phi^{-\epsilon}(c_1)$ because $\epsilon$ vanishes on nonzero graded Reeb chords.
Since we assumed true for $m-1$, the last term becomes $$\sum_{i=2}^m (-1)^{|c_1 \cdots c_{i-1}|}\phi^{-\epsilon}(c_1 \cdots c_{i-1})\partial(\phi^{-\epsilon}(c_i))\phi^{-\epsilon}(c_{i+1} \cdots c_m),$$ and so we get the desired equality.
\\
\indent Now we prove the lemma. We compute that
\begin{equation*}
    \begin{split}
\partial_\epsilon(c_1...c_m) &= \phi^\epsilon \circ \partial \circ \phi^{-\epsilon}(c_1...c_m) \\ &= \phi^\epsilon \circ (\sum_{i=1}^m (-1)^{|c_1 \cdots c_{i-1}|}\phi^{-\epsilon}(c_1 \cdots c_{i-1})\partial(\phi^{-\epsilon}(c_i))\phi^{-\epsilon}(c_{i+1} \cdots c_m) \\ &= \sum_{i=1}^m (-1)^{|c_1 \cdots c_{i-1}|} \phi^{\epsilon} \circ \phi^{-\epsilon} (c_1 \cdots c_{i-1}) \partial_{\epsilon}(c_i) \phi^{\epsilon} \circ \phi^{-\epsilon} (c_{i+1} \cdots c_m) \\ &= \partial^\epsilon(c_1)c_2 \cdots c_m + (-1)^{|c_1|}c_1\partial_\epsilon(c_2)c_3 \cdots c_m + \ldots + (-1)^{|c_1 \cdots c_{m-1}|}c_1 \cdots c_{m-1}\partial_\epsilon(c_m)
\end{split}
\end{equation*}
\end{proof}
We now consider a linearized version of $\Phi_L^{\epsilon_1}$. Let $(F_i, \partial_{\epsilon_i}^{\ell})$ denote the linearized chain groups as defined in Section~\ref{Sec:CE_DGA}. We define the \dfn{linearized augmented DGA map} to be $$(\Phi_L^{\epsilon_1})^\ell: (F_2, \partial_{\epsilon_2}^\ell) \rightarrow (F_1, \partial_{\epsilon_1}^\ell),$$ where $(\Phi_L^{\epsilon_1})^\ell(c)$ for a generator $c$ of $F_2$ is the linear terms in $\Phi_L^{\epsilon_1}(c)$. Now we show that $(\Phi_L^{\epsilon_1})^\ell$ is a DGA map.

\begin{proof}[Proof of Proposition~\ref{DGA}] Let $c$ be a Reeb chord in $\Lambda_2$. We begin by observing that the linear part of $\Phi_L^{\epsilon_1}(\partial_{\epsilon_2}(c))$ is the linear part of the image of the linear part of $\partial_{\epsilon_2}(c)$ under $\Phi_L^{\epsilon_1}$. Let $\mathcal{C}_{\Lambda_1} = \{c_1, \ldots,c_m\}$ be the Reeb chords of $\Lambda_1$, and $a_{i_1, \ldots , i_j} \in \Z$. Recall that $\partial_{\epsilon_1}(c)$ has no constant terms. If $$\partial_{\epsilon_2}(c) = \sum_{j=1}^\infty \left(\sum_{i_1, \ldots, i_j} a_{i_1, \ldots, i_j} c_{i_1} \cdots c_{i_j} \right),$$ with only finitely many $i_k$ nonzero, then $$\Phi_L^{\epsilon_1}(\partial_{\epsilon_2}(c)) = \sum_{j=1}^\infty \left(\sum_{i_1, \ldots, i_j} a_{i_1, \ldots, i_j} \Phi_L^{\epsilon_1}(c_{i_1}) \cdots \Phi_L^{\epsilon_1}(c_{i_j}) \right).$$ We know that for $j>1$, $\Phi_L^{\epsilon_1}(c_{i_1}) \cdots \Phi_L^{\epsilon_1}(c_{i_j})$ must be a sum of words all of length greater than 1 due to Lemma~\ref{NoConstants}. So all of the linear parts of $\Phi_L^{\epsilon_1}(\partial_{\epsilon_2}(c))$ must come from terms when $j=1$, i.e. the linear part of $\partial_{\epsilon_2}(c)$. We have just shown that $$(\Phi_L^{\epsilon_1} \circ \partial_{\epsilon_1})^\ell(c) = (\Phi_L^{\epsilon_2})^\ell \circ \partial_{\epsilon_2}^\ell(c).$$  \\
\indent Now let $$\Phi_L^{\epsilon_1}(c) = \sum_{j=1}^\infty \left(\sum_{i_1, \ldots, i_j} a'_{i_1, \ldots, i_j} c'_{i_1} \cdots c'_{i_j} \right),$$ with only finitely many $i_k$ nonzero, where $c'_{i_1}, \ldots, c'_{i_j} \in \mathcal{C}_{\Lambda_1}$ and $a'_{i_1, \ldots, i_j} \in \Z$. Then taking the augmented differential $\partial_{\epsilon_1}$ of both sides gives $$\partial_{\epsilon_1}(\Phi_L^{\epsilon_1}(c)) = \sum_{j=1}^\infty \left(\sum_{i_1, \ldots, i_j} a'_{i_1, \ldots, i_j} \partial_{\epsilon_1}(c'_{i_1} \cdots c'_{i_j}) \right).$$ By Lemma~\ref{Leibniz} any word of length $j \geq 2$ will be mapped to a sum of words each of which has $j-1$ Reeb chords multiplied by some $\partial_{\epsilon_1}(c_i)$. Since $\partial_{\epsilon_1}(c_i)$ has no constant terms, these words are length at least 2. Thus the linear part of $\partial_{\epsilon_1}(\Phi_L^{\epsilon_1}(c))$ is the linear part of the image of the linear part of $\Phi_L^{\epsilon_1}(c)$ under $\partial_{\epsilon_1}$, and we conclude that $$(\partial_{\epsilon_1} \circ \Phi_L^{\epsilon_1})^\ell(c) = \partial_{\epsilon_1}^\ell \circ (\Phi_L^{\epsilon_1})^\ell(c).$$
\\
\indent Using the facts proven in the previous paragraphs and Lemma~\ref{AugCommute}, we have 
\begin{equation*}
    \begin{split}
\partial_{\epsilon_1}^\ell \circ (\Phi_L^{\epsilon_1})^\ell(c) &= (\partial_{\epsilon_1} \circ \Phi_L^{\epsilon_1})^\ell(c) \\ &= (\Phi_L^{\epsilon_1} \circ \partial_{\epsilon_2})^\ell(c) \\ &= (\Phi_L^{\epsilon_1})^\ell \circ \partial_{\epsilon_2}^\ell(c).
\end{split}
\end{equation*}
Thus $(\Phi_L^{\epsilon_1})^\ell$ is a chain map.
\end{proof}

Now that we have shown the linearized augmented DGA map $(\Phi_L^{\epsilon_1})^\ell$ is indeed a well-defined DGA map, we would like to consider how different geometric phenomena, such as Lagrangian isotopy, affect $(\Phi_L^{\epsilon_1})^\ell$. We do so by proving a linearized version of Theorem~\ref{EHK} parts 1) and 2). Let $\mathcal{C}_{\Lambda_2}$ be the set of Reeb chords of $\Lambda_2$. In \cite{EckholmHondaKalman}, the chain homotopy in part 2) of Theorem~\ref{EHK} is defined on words of Reeb chords as the $\Z[t,t^{-1}]$-linear map $$\Omega(c_1...c_m) = \sum_{j=1}^m \Phi_{L_1}(c_1...c_{j-1}) K(c_j) \Phi_{L_2}(c_{j+1}...c_m)$$ where $c_i \in \mathcal{C}_{\Lambda_2}$ for $1 \leq i \leq m$ and $K$ is a degree 1 map that maps $c_j$ to a finite sum of words in Reeb chords of degree $|c_j|+1$. For our Theorem~\ref{LinearizedEHK}, we define a linearized version of $\Omega$ to use as our chain homotopy in the following lemma. In order to do this we also define a linearized version of $K$ by $K^{\epsilon_1} = \phi^{\epsilon_1} \circ K \circ \phi^{-\epsilon_2}$.

\begin{lemma}
\label{AugmentedOmega} Let $c_1 \cdots c_m$ be a word in $(\mathcal{A}_{\Lambda_2}, \partial_{\Lambda_2})$, and define $\Omega^{\epsilon_1}$ by $$\Omega^{\epsilon_1}(c_1 \cdots c_m) = \sum_{j=1}^m \Phi_{L_1}^{\epsilon_1}(c_1 \cdots c_{j-1}) K^{\epsilon_1}(c_j) \Phi_{L_2}^{\epsilon_1}(c_{j+1} \cdots c_m).$$ Then $$\Omega^{\epsilon_1}(c_1 \cdots c_m) = \phi^{\epsilon_1} \circ \Omega \circ \phi^{-\epsilon_2} (c_1 \cdots c_m).$$
\begin{proof}
Let $c_i \in \mathcal{C}_{\Lambda_2}$ for $1 \leq i \leq m$. First note that 
\begin{equation*}
    \begin{split}
        \Omega^{\epsilon_1}(c_1 \cdots c_m) &= \sum_{j=1}^m (\phi^{\epsilon_1} \circ \Phi_{L_1} \circ \phi^{-\epsilon_2}(c_1 \cdots c_{j-1})) (\phi^{\epsilon_1} \circ K \circ \phi^{-\epsilon_2} (c_j))(\phi^{\epsilon_1} \circ \Phi_{L_2} \circ \phi^{-\epsilon_2}(c_{j+1} \cdots c_m)) \\  &= \sum_{j=1}^m \phi^{\epsilon_1} \circ (\Phi_{L_1} \circ \phi^{-\epsilon_2}(c_1 \cdots c_{j-1}) K \circ \phi^{-\epsilon_2}(c_j) \Phi_{L_2} \circ \phi^{-\epsilon_2}(c_{j+1} \cdots c_m)) \\ &= \phi^{\epsilon_1} \circ (\sum_{j=1}^m \Phi_{L_1} \circ \phi^{-\epsilon_2}(c_1 \cdots c_{j-1}) K \circ \phi^{-\epsilon_2}(c_j) \Phi_{L_2} \circ \phi^{-\epsilon_2}(c_{j+1} \cdots c_m)). 
    \end{split}
\end{equation*}
So to prove the lemma we only need to show $$\Omega \circ \phi^{-\epsilon_2}(c_1 \cdots c_m) = \sum_{j=1}^m \Phi_{L_1} \circ \phi^{-\epsilon_2}(c_1 \cdots c_{j-1}) K \circ \phi^{-\epsilon_2} (c_j) \Phi_{L_2} \circ \phi^{-\epsilon_2}(c_{j+1} \cdots c_m).$$ We do so by showing any word on the left exists uniquely in the sum on the right, and vice versa. \\
\indent Let $$\phi^{-\epsilon_2}(c_1 \cdots c_m) = c_1 \cdots c_m + \sum_{i=1}^m c_1 \cdots \hat{c}_i \cdots c_m + \sum_{i=1}^{m-1} \sum_{k>i}^m c_1 \cdots \hat{c}_i \cdots \hat{c}_k \cdots c_m + \ldots + \hat{c}_1 \cdots \hat{c}_m$$ where $\hat{c}_i = -\epsilon_2(c_i)$. This is a sum of words of length at most $m$. Let $w_n$ be a word in this sum with $n \leq m$ hatted Reeb chords. Then $$\Omega(w_n) = (-1)^n \epsilon_2(c_{i_1} \cdots c_{i_n}) \sum_{j=1}^{m-n} \Phi_{L_1} (c_{i_{n+1}} \cdots c_{i_{n+j-1}})K(c_{i_{n+j}}) \Phi_{L_2}(c_{i_{n+j+1}} \cdots c_{i_m}).$$ Now fix $j$. Let $i_r < i_{n+j}$ be the largest index such that it's Reeb chord is hatted, and $i_s$, $i_{n+j} < i_s$, be the smallest such that its Reeb chord is hatted as well, with $r+s=n$. Then $(-1)^r \epsilon_2(c_{i_1} \cdots c_{i_r})\Phi_{L_1} (c_{i_{n+1}} \cdots c_{i_{n+j-1}})$ is a unique term in $\Phi_{L_1} \circ \phi^{-\epsilon_2} (c_1 \cdots c_{i_{n+j}-1})$, while $K(c_{i_{n+j}}) = K \circ \phi^{-\epsilon_2} (c_{i_{n+j}})$, and $(-1)^s\epsilon_2(c_{i_s} \cdots c_m)\Phi_{L_2}(c_{i_{n+j+1}} \cdots c_m)$ is a unique term in $\Phi_{L_2} \circ \phi^{-\epsilon_2} (c_{i_{n+j}+1} \cdots c_m)$. Therefore their concatenation is a unique word in $$\Phi_{L_1} \circ \phi^{-\epsilon_2} (c_1 \cdots c_{i_{n+j}-1})K \circ \phi^{-\epsilon_2} (c_{i_{n+j}})\Phi_{L_2} \circ \phi^{-\epsilon_2} (c_{i_{n+j}+1} \cdots c_m).$$ Thus any word in $\Omega \circ \phi^{-\epsilon_2} (c_1 \cdots c_m)$ is a unique word in $$\sum_{j=1}^m \Phi_{L_1} \circ \phi^{-\epsilon_2}(c_1 \cdots c_{j-1}) K \circ \phi^{-\epsilon_2} (c_j) \Phi_{L_2} \circ \phi^{-\epsilon_2} (c_{j+1} \cdots c_m).$$ \\
\indent Now consider a word in the sum $\sum_{j=1}^m \Phi_{L_1} \circ \phi^{-\epsilon_2}(c_1 \cdots c_{j-1}) K \circ \phi^{-\epsilon_2} (c_j) \Phi_{L_2} \circ \phi^{-\epsilon_2} (c_{j+1} \cdots c_m)$. For a fixed $j$, we have 
\begin{equation*}
    \begin{split}
&\Phi_{L_1} \circ \phi^{-\epsilon_2} (c_1 \cdots c_{j-1}) K \circ \phi^{-\epsilon_2} (c_j) \Phi_{L_2} \circ \phi^{-\epsilon_2} (c_{j+1} \cdots c_m) \\ &=  (-1)^r\epsilon_2(c_{i_1} \cdots c_{i_r})\Phi_{L_1}(c_1 \cdots c_{j-1})K(c_j)(-1)^{s-r}\epsilon_2(c_{i_{r+1}} \cdots c_s)\Phi_{L_2}(c_{j+1} \cdots c_m) \\ &= (-1)^s\epsilon_2(c_{i_1} \cdots c_{i_s})\Phi_{L_1}(c_1 \cdots c_{j-1})K(c_j)\Phi_{L_2}(c_{j+1} \cdots c_m), 
    \end{split}
\end{equation*}
which is a unique word in the sum $(-1)^s\epsilon_2(c_{i_1} \cdots c_{i_s}) \Omega (w_s)$, where $w_s$ is the word of unhatted Reeb chords in some word in $\phi^{-\epsilon_2}(c_1 \cdots c_m)$ whose hatted Reeb chords are $c_{i_1}, \ldots ,c_{i_s}$. Thus any word in $\sum_{j=1}^m \Phi_{L_1} \circ \phi^{-\epsilon_2}(c_1 \cdots c_{j-1}) K \circ \phi^{-\epsilon_2} (c_j) \Phi_{L_2} \circ \phi^{-\epsilon_2} (c_{j+1} \cdots c_m)$ is a unique word in $\Omega \circ \phi^{-\epsilon_2} (c_1 \cdots c_m)$, and so the lemma follows.
\end{proof}
\end{lemma}
Now that we have shown that $\Omega^{\epsilon_1} = \phi^{\epsilon_1} \circ \Omega \circ \phi^{-\epsilon_2}$, we can prove the linearized version of parts 1) and 2) of Theorem~\ref{EHK}. We remind the reader that the exponent $\ell$ means taking the linear terms of the expression.

\begin{proof}[Proof of Theorem~\ref{LinearizedEHK}]

For Item 1) in the theorem, let $c$ be a Reeb chord of $\Lambda$. Then
\begin{equation*}
    \begin{split}
    (\Phi_L^{\epsilon_1})^\ell(c) &= (\phi^{\epsilon_1} \circ \Phi_L \circ \phi^{-\epsilon_2})^\ell(c) \\ &= (\phi^{\epsilon_1} \circ \Phi_L)^\ell(c - \epsilon_2(c)) \\ &= (\phi^{\epsilon_1})^\ell (c - \epsilon_2(c)) \\ &= (c + \epsilon_1(c) - \epsilon_2(c))^\ell \\ &= c.
\end{split}
\end{equation*}
\indent To prove Statement 2), we start with the homotopy equation from \cite{EckholmHondaKalman}. $$\Phi_{L_1} - \Phi_{L_2} = \Omega \circ \partial_2 + \partial_1 \circ \Omega,$$ which is equivalent to $$\Phi_{L_1}^{\epsilon_1} - \Phi_{L_2}^{\epsilon_1} = \phi^{\epsilon_1} \circ \Omega \circ \phi^{-\epsilon_2} \circ \phi^{\epsilon_2} \circ \partial_2 \circ \phi^{-\epsilon_2} + \phi^{\epsilon_1} \circ \partial_1 \circ \phi^{-\epsilon_1} \circ \phi^{\epsilon_1} \circ \Omega \circ \phi^{-\epsilon_2}.$$ By Lemma~\ref{AugmentedOmega}, this is equivalent to $$\Phi_{L_1}^{\epsilon_1} - \Phi_{L_2}^{\epsilon_1} = \Omega^{\epsilon_1} \circ \partial_{\epsilon_2} + \partial_{\epsilon_1} \circ \Omega^{\epsilon_1}.$$ Linearizing both sides yields $$(\Phi_{L_1}^{\epsilon_1})^\ell - (\Phi_{L_2}^{\epsilon_1})^\ell = (\Omega^{\epsilon_1} \circ \partial_{\epsilon_2})^\ell + (\partial_{\epsilon_1} \circ \Omega^{\epsilon_1})^\ell.$$ We need to show that the linear terms of the compositions $\Omega^{\epsilon_1} \circ \partial_{\epsilon_2}$ and $\partial_{\epsilon_1} \circ \Omega^{\epsilon_1}$ equal the composition of their linearizations, respectively. Let $a_{i_1, \ldots, i_j} \in \Z$ and $\mathcal{C}_{\Lambda_2} = \{c_1, \ldots, c_n\}$ be Reeb chords in $\Lambda_2$ such that $$w = \sum_{j=1}^\infty \left(\sum_{i_1, \ldots, i_j} a_{i_1, \ldots, i_j} c_{i_1} \cdots c_{i_j} \right) \in  \Ima \partial_{\epsilon_2}.$$ We must show words of length greater than 1 are mapped to words also of length greater than 1. We have that $$\Omega^{\epsilon_1} (w) = \sum_{j=1}^\infty \left(\sum_{i_1, \ldots, i_j} a_{i_1, \ldots, i_j} \Omega^{\epsilon_1}(c_{i_1} \cdots c_{i_j}) \right).$$ The image of words of Reeb chords of length greater than 2 under $\Omega^{\epsilon_1}$ will be a sum of words that each have $\Phi_{L_i}^{\epsilon_1}$ applied to at least two of the letters, which have no constant terms by Lemma~\ref{NoConstants}. Thus those words will be length greater than 1. Now consider the images of the quadratic terms $a_{i_1,i_2} \Omega^{\epsilon_1} (c_{i_1} c_{i_2}) = a_{i_1,i_2}(\Phi_{L_1}^{\epsilon_1} (c_{i_1}) K^{\epsilon_1}(c_{i_2}) + K^{\epsilon_1}(c_{i_1}) \Phi_{L_2}^{\epsilon_1}(c_{i_2}))$. Since none of the Reeb chords in $\Ima \partial_{\epsilon_2}$ have grading $-1$, $K \circ \phi^{-\epsilon_2}(c_{i_1})$ and $K \circ \phi^{-\epsilon_2}(c_{i_2})$ are both sums of words that have nonzero grading. Therefore all of those words in their sums must have a nonzero graded Reeb chord, and since $\phi^{\epsilon_1}$ is the identity on generators with nonzero grading, $\phi^{\epsilon_1} \circ K \circ \phi^{-\epsilon_2}(c_{i_1})$ and $\phi^{\epsilon_1} \circ K \circ \phi^{-\epsilon_2}(c_{i_2})$ have no constant terms. We have just shown that $$(\Omega^{\epsilon_1} \circ \partial_{\epsilon_2})^\ell = (\Omega^{\epsilon_1})^\ell \circ \partial_{\epsilon_2}^\ell.$$ \\
By the same argument used in the proof of Proposition~\ref{DGA}, $(\partial_{\epsilon_1} \circ \Omega^{\epsilon_1})^\ell = \partial_{\epsilon_1}^\ell \circ (\Omega^{\epsilon_1})^\ell$, and so we have the chain homotopy equation: $$(\Phi_{L_1}^{\epsilon_1})^\ell - (\Phi_{L_2}^{\epsilon_1})^\ell = (\Omega^{\epsilon_1})^\ell \circ \partial_{\epsilon_2}^\ell + \partial_{\epsilon_1}^\ell \circ (\Omega^{\epsilon_1})^\ell.$$
\end{proof}

\begin{remark}
\label{LinearizedThmRemark}
    In the proof of statement 2), we impose the condition that the quadratics $c_i c_j$ of $\Ima \partial_{\epsilon_2}$ not have any Reeb chords of grading $-1$ because $K(c_j)$ could have a word of Reeb chords in its sum that all have grading zero and are mapped to a nonzero constant by $\epsilon_1$. This in turn would result in a contant term in the sum $K^{\epsilon_1}(c_j)$, and thus $\Phi_{L_1}^{\epsilon_1} (c_i)K^{\epsilon_1}(c_j)$ could have a linear term, obstucting the desired equality $(\Omega^{\epsilon_1} \circ \partial_{\epsilon_2})^\ell = (\Omega^{\epsilon_1})^\ell \circ (\partial_{\epsilon_2})^\ell$.
\end{remark}

\section{Preserving the $A_\infty$ structure}

Let $((\Phi_L^{\epsilon_1})^k)^*$ be the adjoint of the map $(\Phi_L^{\epsilon_1})^k: F_{\Lambda_2} \rightarrow F_{\Lambda_1}^{\otimes k}$ that sends Reeb chords $c$ to the length $k$ words in the sum $\Phi_L^{\epsilon_1}(c)$. We will show the collection $\{((\Phi_L^{\epsilon_1})^k)^*\}_{k \geq 1}$ is an $A_\infty$ morphism. This was shown by Pan in \cite{Pan}, but we will show it using an argument on the chain level similar to how we did in Section~\ref{Sec:cobordism_map}. We first prove the following two lemmas. We remind the reader that $\partial_{\epsilon_i}^k(c)$ denotes the length $k$ terms in the sum $\partial_{\epsilon_i}(c)$, $i = 1,2$. For this section we are working in $\Z/2\Z$ coefficients.

\begin{lemma}
Let $L$ be an exact Lagrangian cobordisms from $\Lambda_1$ to $\Lambda_2$, and assume $j < k$. Then $$(\partial_{\epsilon_1} \circ (\Phi_L^{\epsilon_1})^j)^k (c) = \sum_{i+n+1 = j}(1^{\otimes i} \otimes \partial_{\epsilon_1}^{k-j+1} \otimes 1^{\otimes n}) \circ (\Phi_L^{\epsilon_1})^j(c)$$
\end{lemma}

\begin{proof}
Let $$(\Phi_L^{\epsilon_1})^j(c) = \sum_{i_1, \ldots, i_j} a_{i_1, \ldots, i_j} c_{i_1} \cdots c_{i_j}.$$ 
Then 
\begin{equation*}
\begin{split}
\partial_{\epsilon_1} \circ (\Phi_L^{\epsilon_1})^j(c) &= \sum_{i_1, \ldots, i_j} a_{i_1, \ldots, i_j} \partial_{\epsilon_1}(c_{i_1} \cdots c_{i_j}) \\ &= \sum_{i_1, \ldots, i_j} a_{i_1, \ldots, i_j} (\partial_{\epsilon_1}(c_{i_1})c_{i_2} \cdots c_{i_j} + \ldots + c_{i_1} c_{i_2} \cdots \partial_{\epsilon_1}(c_{i_j})).
\end{split}
\end{equation*}
The length $k$ words of this sum will be exactly $$\sum_{i_1, \ldots, i_j} a_{i_1, \ldots, i_j}(\partial_{\epsilon_1}^{k-j+1}(c_{i_1})c_{i_2} \cdots c_{i_j} + \ldots + c_{i_1} c_{i_2} \cdots \partial_{\epsilon_1}^{k-j+1}(c_{i_j})),$$
or
$$\sum_{i_1, \ldots, i_j} a_{i_1, \ldots, i_j} (\partial_{\epsilon_1}^{k-j+1} \otimes 1 \otimes \cdots \otimes 1 + \ldots + 1 \otimes \cdots \otimes 1 \otimes \partial_{\epsilon_1}^{k-j+1})(c_{i_1} \cdots c_{i_j}),$$
which can be rewritten as $$\sum_{i+1+n = j} \sum_{i_1, \ldots, i_j} a_{i_1, \ldots, i_j} (1^{\otimes i} \otimes \partial_{\epsilon_1}^{k-j+1} \otimes 1^{\otimes n})(c_{i_1} \cdots c_{i_j}),$$
and the lemma follows.
\end{proof}

\begin{lemma}
Let $L$ be an exact Lagrangian cobordisms from $\Lambda_1$ to $\Lambda_2$, and assume $j < k$. Then $$((\Phi_L^{\epsilon_1}) \circ \partial_{\epsilon_2}^j)^k(c) = \sum_{k_1+...+k_j = k}((\Phi_L^{\epsilon_1})^{k_1} \otimes ... \otimes (\Phi_L^{\epsilon_1})^{k_j}) \circ \partial_{\epsilon_2}^j(c)$$. 
\end{lemma}

\begin{proof}
Let $$\partial_{\epsilon_2}^j (c) = \sum_{i_1, \ldots, i_j} a_{i_1,\ldots,i_j} c_{i_1}c_{i_2} \cdots c_{i_j}.$$ 
Then $$\Phi_L^{\epsilon_1} \circ \partial_{\epsilon_2}^j (c) = \sum_{i_1, \ldots, i_j} a_{i_1, \ldots ,i_j} \Phi_L^{\epsilon_1} (c_{i_1}c_{i_2} \cdots c_{i_j}) = \sum_{i_1, \ldots, i_j} a_{i_1, \ldots ,i_j} \Phi_L^{\epsilon_1}(c_{i_1}) \cdots \Phi_L^{\epsilon_1}(c_{i_j}).$$
Then length $k$ words coming from this product will be exactly those resulting from $$(\Phi_L^{\epsilon_1})^{k_1}(c_{i_1})(\Phi_L^{\epsilon_1})^{k_2}(c_{i_2}) \cdots (\Phi_L^{\epsilon_1})^{k_j}(c_{i_j}),$$ 
where $k_1+ \ldots + k_j = k$. Summing over all such $k_1, \ldots, k_j$ proves the lemma.
\end{proof}

We now show $\{((\Phi_L^{\epsilon_1})^k)^*\}_{k \geq 1}$ is an $A_\infty$ morphism. We denote the adjoint of $\partial_{\epsilon_i}^k$ by $m_k^{\epsilon_i}$, $i=1,2$.

\begin{proof}[Proof of Theorem~\ref{AinftyTHM}]

We first consider the chain groups and then dualize. Consider the equation $$(\partial_{\epsilon_1} \circ \Phi_L^{\epsilon_1})^k(c) = (\Phi_L^{\epsilon_1} \circ \partial_{\epsilon_2})^k(c).$$
This can be rewritten $$ \sum_{j=1}^k (\partial_{\epsilon_1} \circ (\Phi_L^{\epsilon_1})^j)^k (c) = \sum_{j=1}^k ((\Phi_L^{\epsilon_1}) \circ \partial_{\epsilon_2}^j)^k (c)$$
because if $j>k$, then the Leibniz rule for $\partial_{\epsilon_1}$ and $\Phi_L^{\epsilon_1}$ being a homomorphism combined with there being no constants in their images imply that a word of length $j>k$ would be sent to a sum of words of length greater than $k$. By the lemmas above, this can be rewritten $$\sum_{j=1}^k \sum_{i+n+1 = j}(1^{\otimes i} \otimes \partial_{\epsilon_1}^{k-j+1} \otimes 1^{\otimes n}) \circ (\Phi_L^{\epsilon_1})^j (c) = \sum_{j=1}^k \sum_{k_1+...+k_j = k}((\Phi_L^{\epsilon_1})^{k_1} \otimes \cdots \otimes (\Phi_L^{\epsilon_1})^{k_j}) \circ \partial_{\epsilon_2}^j(c).$$
Dualizing yields the equation
$$\sum_{j=1}^k \sum_{i+n+1=j} ((\Phi_L^{\epsilon_1})^j)^* \circ (1^{\otimes i} \otimes m_{k-j+1}^{\epsilon_1} \otimes 1^{\otimes n}) = \sum_{j=1}^k \sum_{k_1+...+k_j = k} m_j^{\epsilon_2} \circ (((\Phi_L^{\epsilon_1})^{k_1})^* \otimes \cdots \otimes ((\Phi_L^{\epsilon_1})^{k_j})^*).$$
Reindexing the left side gives the equation
$$\sum_{i+p+n=k} ((\Phi_L^{\epsilon_1})^{i+n+1})^* \circ (1^{\otimes i} \otimes m_p^{\epsilon_1} \otimes 1^{\otimes n}) = \sum_{j=1}^k \sum_{k_1+...+k_j = k} m_j^{\epsilon_2} \circ (((\Phi_L^{\epsilon_1})^{k_1})^* \otimes \cdots \otimes ((\Phi_L^{\epsilon_1})^{k_j})^*).$$
Hence $\{((\Phi_L^{\epsilon_1})^k)^*: (F_1^*)^{\otimes k} \rightarrow F_2^*\}_{k \geq 1}$ is an $A_\infty$ morphism.

\end{proof}

Now that we have shown an $A_\infty$ morphism between $(F_1^*, \{m_k^{\epsilon_1}\}_{k \geq 1})$ and $(F_2^*, \{m_k^{\epsilon_2}\}_{k \geq 1})$, let us consider the induced maps on their higher order product structures by two Lagrangian isotopic cobordisms. If two exact Lagrangian cobordisms from $\Lambda_1$ to $\Lambda_2$ are Lagrangian isotopic (and satisfy the conditions in Theorem~\ref{LinearizedEHK}), then we can dualize the chain homotopy equation from Theorem~\ref{LinearizedEHK} to give us a chain homotopy equation between $((\Phi_{L_1}^{\epsilon_1})^\ell)^*$ and $((\Phi_{L_2}^{\epsilon_2})^\ell)^*$. Since the product structures on $(F_1^*, \{m_k^{\epsilon_1}\}_{k \geq 1})$ and $(F_2^*, \{m_k^{\epsilon_2}\}_{k \geq 1})$ are defined on quotients of the cohomologies of $(F_1^*, (\partial_{\epsilon_1}^\ell)^*)$ and $(F_2^*, (\partial_{\epsilon_2}^\ell)^*)$, the chain homotopy equation between $((\Phi_{L_1}^{\epsilon_1})^\ell)^*$ and $((\Phi_{L_2}^{\epsilon_2})^\ell)^*$ implies their induced maps on these product structures are equal. Thus we have the following corollary.

\begin{corollary}
\label{corollary}
    Let $L_1$ be Lagrangian isotopic to $L_2$, $\epsilon_1$ an augmentation for $\Lambda_1$, $\epsilon_1 \circ \Phi_{L_1} = \epsilon_1 \circ \Phi_{L_2}$, and the quadratics in the image of $\partial_{\epsilon_2}$ have no degree $-1$ Reeb chords. Then the induced maps of $((\Phi_{L_1}^{\epsilon_1})^\ell)^*$ and  $((\Phi_{L_2}^{\epsilon_1})^\ell)^*$ are equal on the higher order product structures. 
\end{corollary}

\bibliographystyle{plain} 
\bibliography{cited}

@incollection {Blackwell_legout_etal_21survey,
    AUTHOR = {Blackwell, Sarah and Legout, No\'emie and Leverson, Caitlin
              and Limouzineau, Ma\"ylis and Myer, Ziva and Pan, Yu and
              Pezzimenti, Samantha and Su\'arez, Lara Simone and Traynor,
              Lisa},
     TITLE = {Constructions of {L}agrangian cobordisms},
 BOOKTITLE = {Research directions in symplectic and contact geometry and
              topology},
    SERIES = {Assoc. Women Math. Ser.},
    VOLUME = {27},
     PAGES = {245--272},
 PUBLISHER = {Springer, Cham},
      YEAR = {[2021] \copyright 2021},
      ISBN = {978-3-030-80978-2; 978-3-030-80979-9},
   MRCLASS = {53D12 (57R90)},
  MRNUMBER = {4417718},
MRREVIEWER = {Lenhard\ L.\ Ng},
       DOI = {10.1007/978-3-030-80979-9\_5},
       URL = {https://doi.org/10.1007/978-3-030-80979-9_5},
}

@article {BourgeoisSabloffTraynor,
    AUTHOR = {Bourgeois, Fr\'ed\'eric and Sabloff, Joshua M. and Traynor,
              Lisa},
     TITLE = {Lagrangian cobordisms via generating families: construction
              and geography},
   JOURNAL = {Algebr. Geom. Topol.},
  FJOURNAL = {Algebraic \& Geometric Topology},
    VOLUME = {15},
      YEAR = {2015},
    NUMBER = {4},
     PAGES = {2439--2477},
      ISSN = {1472-2747,1472-2739},
   MRCLASS = {57R17 (53D12 53D42)},
  MRNUMBER = {3402346},
MRREVIEWER = {Lenhard\ L.\ Ng},
       DOI = {10.2140/agt.2015.15.2439},
       URL = {https://doi.org/10.2140/agt.2015.15.2439},
}

@article {Chantraine_10,
    AUTHOR = {Chantraine, Baptiste},
     TITLE = {Lagrangian concordance of {L}egendrian knots},
   JOURNAL = {Algebr. Geom. Topol.},
  FJOURNAL = {Algebraic \& Geometric Topology},
    VOLUME = {10},
      YEAR = {2010},
    NUMBER = {1},
     PAGES = {63--85},
      ISSN = {1472-2747,1472-2739},
   MRCLASS = {57R17 (53D12 57M50)},
  MRNUMBER = {2580429},
MRREVIEWER = {Lenhard\ L.\ Ng},
       DOI = {10.2140/agt.2010.10.63},
       URL = {https://doi.org/10.2140/agt.2010.10.63},
}

@article {Chekanov_DGA_02,
    AUTHOR = {Chekanov, Yuri},
     TITLE = {Differential algebra of {L}egendrian links},
   JOURNAL = {Invent. Math.},
  FJOURNAL = {Inventiones Mathematicae},
    VOLUME = {150},
      YEAR = {2002},
    NUMBER = {3},
     PAGES = {441--483},
      ISSN = {0020-9910,1432-1297},
   MRCLASS = {53D35 (57M27 57R17)},
  MRNUMBER = {1946550},
MRREVIEWER = {John\ B.\ Etnyre},
       DOI = {10.1007/s002220200212},
       URL = {https://doi.org/10.1007/s002220200212},
}

@article {CivanEtnyreKoprowskiSabloffWalker,
    AUTHOR = {Civan, Gokhan and Koprowski, Paul and Etnyre, John and
              Sabloff, Joshua M. and Walker, Alden},
     TITLE = {Product structures for {L}egendrian contact homology},
   JOURNAL = {Math. Proc. Cambridge Philos. Soc.},
  FJOURNAL = {Mathematical Proceedings of the Cambridge Philosophical
              Society},
    VOLUME = {150},
      YEAR = {2011},
    NUMBER = {2},
     PAGES = {291--311},
      ISSN = {0305-0041,1469-8064},
   MRCLASS = {53D42 (16E45 55S30 57R17)},
  MRNUMBER = {2770065},
MRREVIEWER = {Lenhard\ L.\ Ng},
       DOI = {10.1017/S0305004110000460},
       URL = {https://doi.org/10.1017/S0305004110000460},
}

@article {CornwellNgSivek,
    AUTHOR = {Cornwell, Christopher and Ng, Lenhard and Sivek, Steven},
     TITLE = {Obstructions to {L}agrangian concordance},
   JOURNAL = {Algebr. Geom. Topol.},
  FJOURNAL = {Algebraic \& Geometric Topology},
    VOLUME = {16},
      YEAR = {2016},
    NUMBER = {2},
     PAGES = {797--824},
      ISSN = {1472-2747,1472-2739},
   MRCLASS = {57M25 (53D12 53D42 57R17)},
  MRNUMBER = {3493408},
MRREVIEWER = {Georgios\ Dimitroglou Rizell},
       DOI = {10.2140/agt.2016.16.797},
       URL = {https://doi.org/10.2140/agt.2016.16.797},
}

@article {EkholmEtnyreSullivan,
    AUTHOR = {Ekholm, Tobias and Etnyre, John and Sullivan, Michael},
     TITLE = {Non-isotopic {L}egendrian submanifolds in {$\Bbb R^{2n+1}$}},
   JOURNAL = {J. Differential Geom.},
  FJOURNAL = {Journal of Differential Geometry},
    VOLUME = {71},
      YEAR = {2005},
    NUMBER = {1},
     PAGES = {85--128},
      ISSN = {0022-040X,1945-743X},
   MRCLASS = {53D35 (57R17 57R40)},
  MRNUMBER = {2191769},
MRREVIEWER = {Hansj\"org\ Geiges},
       URL = {http://projecteuclid.org/euclid.jdg/1143644313},
}

@article {EckholmHondaKalman,
    AUTHOR = {Ekholm, Tobias and Honda, Ko and K\'alm\'an, Tam\'as},
     TITLE = {Legendrian knots and exact {L}agrangian cobordisms},
   JOURNAL = {J. Eur. Math. Soc. (JEMS)},
  FJOURNAL = {Journal of the European Mathematical Society (JEMS)},
    VOLUME = {18},
      YEAR = {2016},
    NUMBER = {11},
     PAGES = {2627--2689},
      ISSN = {1435-9855,1435-9863},
   MRCLASS = {53D42 (53D10 53D40 57M27 57R90)},
  MRNUMBER = {3562353},
MRREVIEWER = {Georgios\ Dimitroglou Rizell},
       DOI = {10.4171/JEMS/650},
       URL = {https://doi.org/10.4171/JEMS/650},
}

@article{Eliashberg1998,
    author = {Eliashberg, Yakov},
    journal = {Documenta Mathematica},
    language = {eng},
    pages = {327-338},
    publisher = {Universiät Bielefeld, Fakultät für Mathematik},
    title = {Invariants in contact topology.},
    url = {http://eudml.org/doc/231306},
    year = {1998},
}

@article {EpsteinFuchsMeyer,
    AUTHOR = {Epstein, Judith and Fuchs, Dmitry and Meyer, Maike},
     TITLE = {Chekanov-{E}liashberg invariants and transverse approximations
              of {L}egendrian knots},
   JOURNAL = {Pacific J. Math.},
  FJOURNAL = {Pacific Journal of Mathematics},
    VOLUME = {201},
      YEAR = {2001},
    NUMBER = {1},
     PAGES = {89--106},
      ISSN = {0030-8730,1945-5844},
   MRCLASS = {57M27 (53D10)},
  MRNUMBER = {1867893},
MRREVIEWER = {Serge\ L.\ Tabachnikov},
       DOI = {10.2140/pjm.2001.201.89},
       URL = {https://doi.org/10.2140/pjm.2001.201.89},
}

@incollection {Etnyre_05survey,
    AUTHOR = {Etnyre, John B.},
     TITLE = {Legendrian and transversal knots},
 BOOKTITLE = {Handbook of knot theory},
     PAGES = {105--185},
 PUBLISHER = {Elsevier B. V., Amsterdam},
      YEAR = {2005},
      ISBN = {0-444-51452-X},
   MRCLASS = {57R17 (53D35 57M25 57M27)},
  MRNUMBER = {2179261},
MRREVIEWER = {Lenhard\ L.\ Ng},
       DOI = {10.1016/B978-044451452-3/50004-6},
       URL = {https://doi.org/10.1016/B978-044451452-3/50004-6},
}

@incollection {Etnyre_Ng_22survey,
    AUTHOR = {Etnyre, John B. and Ng, Lenhard L.},
     TITLE = {Legendrian contact homology in {$\Bbb R^3$}},
 BOOKTITLE = {Surveys in differential geometry 2020. {S}urveys in 3-manifold
              topology and geometry},
    SERIES = {Surv. Differ. Geom.},
    VOLUME = {25},
     PAGES = {103--161},
 PUBLISHER = {Int. Press, Boston, MA},
      YEAR = {[2022] \copyright 2022},
      ISBN = {978-1-57146-419-4},
   MRCLASS = {57Kxx (53D35)},
  MRNUMBER = {4479751},
}

@article {EtnyreNgSabloff,
    AUTHOR = {Etnyre, John B. and Ng, Lenhard L. and Sabloff, Joshua M.},
     TITLE = {Invariants of {L}egendrian knots and coherent orientations},
   JOURNAL = {J. Symplectic Geom.},
  FJOURNAL = {The Journal of Symplectic Geometry},
    VOLUME = {1},
      YEAR = {2002},
    NUMBER = {2},
     PAGES = {321--367},
      ISSN = {1527-5256,1540-2347},
   MRCLASS = {57R17 (53D35)},
  MRNUMBER = {1959585},
MRREVIEWER = {Ivan\ Smith},
       DOI = {10.4310/jsg.2001.v1.n2.a5},
       URL = {https://doi.org/10.4310/jsg.2001.v1.n2.a5},
}

@article {Gromov,
    AUTHOR = {Gromov, M.},
     TITLE = {Pseudo holomorphic curves in symplectic manifolds},
   JOURNAL = {Invent. Math.},
  FJOURNAL = {Inventiones Mathematicae},
    VOLUME = {82},
      YEAR = {1985},
    NUMBER = {2},
     PAGES = {307--347},
      ISSN = {0020-9910,1432-1297},
   MRCLASS = {53C15 (32F25 53C57 57R15)},
  MRNUMBER = {809718},
MRREVIEWER = {Yakov\ Eliashberg},
       DOI = {10.1007/BF01388806},
       URL = {https://doi.org/10.1007/BF01388806},
}

@article {GuadagniSabloffYacavone,
    AUTHOR = {Guadagni, Roberta and Sabloff, Joshua M. and Yacavone,
              Matthew},
     TITLE = {Legendrian satellites and decomposable cobordisms},
   JOURNAL = {J. Knot Theory Ramifications},
  FJOURNAL = {Journal of Knot Theory and its Ramifications},
    VOLUME = {31},
      YEAR = {2022},
    NUMBER = {13},
     PAGES = {Paper No. 2250071, 33},
      ISSN = {0218-2165,1793-6527},
   MRCLASS = {57R10 (53D12 57R32)},
  MRNUMBER = {4523297},
MRREVIEWER = {Yu\ Pan},
       DOI = {10.1142/s0218216522500717},
       URL = {https://doi.org/10.1142/s0218216522500717},
}

@article {Lin16,
    AUTHOR = {Lin, Francesco},
     TITLE = {Exact {L}agrangian caps of {L}egendrian knots},
   JOURNAL = {J. Symplectic Geom.},
  FJOURNAL = {The Journal of Symplectic Geometry},
    VOLUME = {14},
      YEAR = {2016},
    NUMBER = {1},
     PAGES = {269--295},
      ISSN = {1527-5256,1540-2347},
   MRCLASS = {57R17 (57M25)},
  MRNUMBER = {3523257},
MRREVIEWER = {Tetsuya\ Ito},
       DOI = {10.4310/JSG.2016.v14.n1.a10},
       URL = {https://doi.org/10.4310/JSG.2016.v14.n1.a10},
}

@article {NgCharAlgebra,
    AUTHOR = {Ng, Lenhard L.},
     TITLE = {Computable {L}egendrian invariants},
   JOURNAL = {Topology},
  FJOURNAL = {Topology. An International Journal of Mathematics},
    VOLUME = {42},
      YEAR = {2003},
    NUMBER = {1},
     PAGES = {55--82},
      ISSN = {0040-9383},
   MRCLASS = {57R17 (57M27)},
  MRNUMBER = {1928645},
MRREVIEWER = {David\ T.\ Gay},
       DOI = {10.1016/S0040-9383(02)00010-1},
       URL = {https://doi.org/10.1016/S0040-9383(02)00010-1},
}

@misc{ng2025linfinitystructurelegendriancontact,
      title={An L-infinity structure for Legendrian contact homology}, 
      author={Lenhard Ng},
      year={2025},
      eprint={2311.14614},
      archivePrefix={arXiv},
      primaryClass={math.SG},
      url={https://arxiv.org/abs/2311.14614}, 
}

@article {NgRutherford,
    AUTHOR = {Ng, Lenhard and Rutherford, Daniel},
     TITLE = {Satellites of {L}egendrian knots and representations of the
              {C}hekanov-{E}liashberg algebra},
   JOURNAL = {Algebr. Geom. Topol.},
  FJOURNAL = {Algebraic \& Geometric Topology},
    VOLUME = {13},
      YEAR = {2013},
    NUMBER = {5},
     PAGES = {3047--3097},
      ISSN = {1472-2747,1472-2739},
   MRCLASS = {57M25 (53D42)},
  MRNUMBER = {3116313},
MRREVIEWER = {Quach thi C\^am V\^an},
       DOI = {10.2140/agt.2013.13.3047},
       URL = {https://doi.org/10.2140/agt.2013.13.3047},
}

@article {Pan,
    AUTHOR = {Pan, Yu},
     TITLE = {The augmentation category map induced by exact {L}agrangian
              cobordisms},
   JOURNAL = {Algebr. Geom. Topol.},
  FJOURNAL = {Algebraic \& Geometric Topology},
    VOLUME = {17},
      YEAR = {2017},
    NUMBER = {3},
     PAGES = {1813--1870},
      ISSN = {1472-2747,1472-2739},
   MRCLASS = {53D42 (53D12 57M50 57R17)},
  MRNUMBER = {3677941},
MRREVIEWER = {Paolo\ Ghiggini},
       DOI = {10.2140/agt.2017.17.1813},
       URL = {https://doi.org/10.2140/agt.2017.17.1813},
}

@article {Sabloff,
    AUTHOR = {Sabloff, Joshua M.},
     TITLE = {Duality for {L}egendrian contact homology},
   JOURNAL = {Geom. Topol.},
  FJOURNAL = {Geometry and Topology},
    VOLUME = {10},
      YEAR = {2006},
     PAGES = {2351--2381},
      ISSN = {1465-3060,1364-0380},
   MRCLASS = {53D35 (57M25 57R17)},
  MRNUMBER = {2284060},
MRREVIEWER = {Dragomir\ L.\ Dragnev},
       DOI = {10.2140/gt.2006.10.2351},
       URL = {https://doi.org/10.2140/gt.2006.10.2351},
}

@misc{ma2025weakrelativecalabiyaustructures,
      title={Weak Relative Calabi-Yau Structures for Legendrian Contact Homology}, 
      author={Jiajie Ma and Joshua M. Sabloff},
      year={2025},
      eprint={2509.02485},
      archivePrefix={arXiv},
      primaryClass={math.SG},
      url={https://arxiv.org/abs/2509.02485}, 
}

@article {SabloffTraynor,
    AUTHOR = {Sabloff, Joshua M. and Traynor, Lisa},
     TITLE = {Obstructions to {L}agrangian cobordisms between {L}egendrians
              via generating families},
   JOURNAL = {Algebr. Geom. Topol.},
  FJOURNAL = {Algebraic \& Geometric Topology},
    VOLUME = {13},
      YEAR = {2013},
    NUMBER = {5},
     PAGES = {2733--2797},
      ISSN = {1472-2747,1472-2739},
   MRCLASS = {53D12 (57Q60 57R17)},
  MRNUMBER = {3116302},
MRREVIEWER = {Jelena\ Kati\'c},
       DOI = {10.2140/agt.2013.13.2733},
       URL = {https://doi.org/10.2140/agt.2013.13.2733},
}

\end{document}